\documentclass[10pt,a4paper,reqno]{amsart}
\makeatletter
\renewcommand\subsection{\@startsection{subsection}{2}
	\z@{.5\linespacing\@plus.7\linespacing}{-.5em}%
	{\normalfont\scshape}}
\makeatother
\usepackage{amssymb,amsmath,amsthm,fancyhdr,graphicx,MnSymbol,latexsym}
\usepackage{charter}

\newcommand{\foot}{\footnote}

\newcommand{\kl}{\left(}

\newcommand{\ml}{\left\{}
\newcommand{\il}{\left[}

\newcommand{\mm}{\,\middle|\,}

\newcommand{\kr}{\right)}

\newcommand{\mr}{\right\}}
\newcommand{\ir}{\right]}

\newcommand{\follows}{\ensuremath{\Rightarrow}}

\newcommand{\ld}{\ensuremath{,\ldots,}}
\newcommand{\ssq}{\ensuremath{\subseteq}}
\newcommand{\smin}{\ensuremath{\setminus}}
\newcommand{\eps}{\ensuremath{\varepsilon}}

\newcommand{\Id}{\ensuremath{\mathrm{Id}}}

\newcommand{\Leb}{\ensuremath{\mathrm{Leb}}}

\newcommand{\inte}{\ensuremath{\mathrm{int}}}

\newcommand{\kreis}{\ensuremath{\mathbb{T}^{1}}}

\newcommand{\sltr}{\ensuremath{\textrm{SL}(2,\mathbb{R})}}
\newcommand{\torus}{\ensuremath{\mathbb{T}^2}}

\newcommand{\diffeo}{\ensuremath{\mathrm{Diffeo}}}
\newcommand{\homeo}{\ensuremath{\mathrm{Homeo}}}

\newcommand{\nfolge}[1]{\ensuremath{(#1)_{n\in\mathbb{N}}}}

\newcommand{\alphlist}{\begin{list}{(\alph{enumi})}{\usecounter{enumi}\setlength{\parsep}{2pt}
      \setlength{\itemsep}{1pt} \setlength{\topsep}{5pt}
      \setlength{\partopsep}{3pt}}}
\newcommand{\arablist}{\begin{list}{(\arabic{enumi})}{\usecounter{enumi}\setlength{\parsep}{2pt}
          \setlength{\itemsep}{1pt} \setlength{\topsep}{5pt}
          \setlength{\partopsep}{3pt}}}
\newcommand{\romanlist}{\begin{list}{(\roman{enumi})}{\usecounter{enumi}\setlength{\parsep}{2pt}
              \setlength{\itemsep}{1pt} \setlength{\topsep}{5pt}
              \setlength{\partopsep}{3pt}}}
\newcommand{\Romanlist}{\begin{list}{(\Roman{enumi})}{\usecounter{enumi}\setlength{\parsep}{2pt}
              \setlength{\itemsep}{1pt} \setlength{\topsep}{5pt}
              \setlength{\partopsep}{3pt}}}
\newcommand{\bulletlist}{\begin{list}{$\bullet$}{\setlength{\parsep}{2pt}
                \setlength{\itemsep}{1pt} \setlength{\topsep}{5pt}
                \setlength{\partopsep}{3pt}\setlength{\leftmargin}{15pt}}} 
\newcommand{\Alphlist}{\begin{list}{(\Alph{enumi})}{\usecounter{enumi}\setlength{\parsep}{2pt}
      \setlength{\itemsep}{1pt} \setlength{\topsep}{5pt}
      \setlength{\partopsep}{3pt}}}
 \newcommand{\listend}{\end{list}}

\newcommand{\T}{\ensuremath{\mathbb{T}}}

\newcommand{\N}{\ensuremath{\mathbb{N}}} 
\newcommand{\R}{\ensuremath{\mathbb{R}}}
\newcommand{\Z}{\ensuremath{\mathbb{Z}}}
\newcommand{\Q}{\ensuremath{\mathbb{Q}}}
\newcommand{\C}{\ensuremath{\mathbb{C}}}

\newcommand{\cA}{\mathcal{A}}
\newcommand{\cB}{\mathcal{B}}
\newcommand{\cC}{\mathcal{C}}

\newcommand{\cE}{\mathcal{E}}
\newcommand{\cF}{\mathcal{F}}

\newcommand{\cJ}{\mathcal{J}}

\newcommand{\cO}{\mathcal{O}}
\newcommand{\cP}{\mathcal{P}}

\newcommand{\ncap}{\ensuremath{\bigcap_{n\in\N}}}

\newcommand{\nLim}{\ensuremath{\lim_{n\rightarrow\infty}}}

\newcommand{\nKonv}{\ensuremath{\stackrel{n\rightarrow
      \infty}{\longrightarrow}}}

\newcommand{\jmsum}{\ensuremath{\sum_{j=1}^m}}

\newcommand{\inergsum}{\ensuremath{\sum_{i=0}^{n-1}}}

\newcommand{\mtel}{\ensuremath{\frac{1}{m}}}
\newcommand{\ntel}{\ensuremath{\frac{1}{n}}}

\newcommand{\drittel}{\ensuremath{\frac{1}{3}}}

\usepackage{papercommands}

\title[Non-almost automorphic isomorphic extensions] {Construction of smooth
  isomorphic and finite-to-one extensions of irrational rotations which are not
  almost automorphic} \author{L.~Haupt \and T.~J\"ager}

\begin{document}

\begin{abstract}
Due to a result by Glasner and Downarowicz, it is known that a minimal system is
mean equicontinuous if and only if it is an isomorphic extension of its maximal
equicontinuous factor. The majority of known examples of this type are almost
automorphic, that is, the factor map to the maximal equicontinuous factor is
almost one-to-one. The only cases of isomorphic extensions which are not almost
automorphic are again due to Glasner and Downarowicz, who in the same article
provide a construction of such systems in a rather general topological setting.

Here, we use the Anosov-Katok method in order to provide an alternative route to
such examples and to show that these may be realised as smooth skew product
diffeomorphisms of the two-torus with an irrational rotation on the
base. Moreover -- and more importantly -- a modification of the construction
allows to ensure that lifts of these diffeomorphism to finite covering spaces
provide novel examples of finite-to-one topomorphic extensions of irrational
rotations. These are still strictly ergodic and share the same dynamical
eigenvalues as the original system, but show an additional singular continuous
component of the dynamical spectrum.\medskip

\noindent{\em 2010 Mathematics Subject Classification.} 37B05 (primary),
37C05 (secondary).
\end{abstract}

\maketitle

\section{Introduction}

The celebrated Halmos-von Neumann Theorem provides a classification, up to
isomorphism, of ergodic measure-preserving dynamical systems with discrete
dynamical spectrum. Moreover, any such system can be realised as a rotation on
some compact abelian group
\cite{VonNeumann1932Operatorenmethode,HalmosVonNeumann1942OperatorMethodsII}. From
the measure-theoretic viewpoint, this provides a rather complete picture for the
class of dynamical systems with discrete spectrum. However, topological
realisations of such systems can still show a surprising variety of different
behaviours. One particular subclass that has recently attracted considerable
attention are mean equicontinuous systems
\cite{LiTuYe2015MeanSensitivity,Garcia-RamosMarcus2015MeanSensitivity,LiYeTu2021MeanEquicontinuityComplexity,HuangLuYe2011MeasureTheoreticSensitivity,GarciaRamosLiZhang2019MeanSensitivity}. In
the minimal case, Downarowicz and Glasner showed that these are exactly those
topological dynamical systems which are measure-theoretically isomorphic to
their maximal equicontinuous factor (MEF) via the respective continuous factor
map \cite{DownarowiczGlasner2015IsomorphicExtensionsAndMeanEquicontinuity}. Such
systems are called isomorphic extensions (of the MEF). Equivalently, these
systems are characterised by discrete spectrum with continuous eigenfunctions. A
generalisation of these results to the non-minimal case and more general group
actions is provided in
\cite{FuhrmannGroegerLenz2022MeanEquicontinuousGroupActions}. Subsequent work
has concentrated on characterising different types of mean equicontinuous
systems in terms of invertibility properties of the factor map to the MEF
(e.g. . For instance, the fact that almost all points of a strictly ergodic
system are injectivity points of the factor map, which implies mean
equicontinuity, is equivalent to the stronger property of diam mean
equicontinuity
\cite{GarciaRamos2017WeakEquicontinuity,GarciaJaegerYe2021DiamMeanEquicontinuity}.

Examples of mean equicontinuous systems in the literature are abundant. In
particular, these include the classes of regular Toeplitz flows
\cite{JacobsKeane1969ToeplitzSequences,Williams1984ToeplitzFlows,%
  MarkleyPaul1979PositiveEntropyToeplitzFlows,Downarowicz2005ToeplitzFlows} and
regular model sets arising from Meyer's cut and project method
\cite{Meyer1972AlgebraicNumbers,Schlottmann1999GeneralizedModelSets,%
  Schlottmann1999GeneralizedModelSets,Moody2000ModelSetsSurvey,%
  BaakeLenzMoody2007Characterization}. In both cases, the factor map is almost
surely injective, so that the dynamics are diam mean equicontinuous. Examples of
mean equicontinuous systems whose factor maps are not almost surely injective
are given by certain irregular Toeplitz flows
(e.g. \cite{Williams1984ToeplitzFlows}) and irregular models sets
\cite{FuhrmannGlasnerJaegerOertel2021TameImpliesRegular}. In these cases, the
systems are almost automorphic, meaning that the factor maps are 
almost one-to-one, i.e.~the set of injectivity points is residual.

Mean equicontinuous systems for which the factor map to the MEF has no singular
fibres are much more difficult to find. In fact, to the best of our knowledge,
the only non-trivial examples\foot{A `trivial' example would be a homeomorphism
  of the circle  with a unique fixed
  point. In this case, the MEF is just a single point.}  were so far given by
Glasner and Downarowicz in
\cite{DownarowiczGlasner2015IsomorphicExtensionsAndMeanEquicontinuity}, who
showed that homeomorphisms with these properties are generic in certain spaces
of extensions of minimal group rotations.  One aim of this note is to provide an
alternative construction of such examples based on the well-known Anosov-Katok
method \cite{anosov/katok:1970,fayad2004constructions}. As a byproduct, we also obtain the smoothness of the resulting
diffeomorphisms.

\begin{thm} \label{t.meanequi_with_full_fibres}
  There exist $\cC^\infty$-diffeomorphisms $\varphi$ of the two-torus with the
  following properties.  \alphlist
  \item $\varphi$ is a skew product over some irrational rotation
    $R_\alpha:\kreis\to\kreis, x\mapsto x+\alpha\bmod 1$.
  \item $\varphi$ is totally strictly ergodic\foot{All iterates of $\varphi$ are
    strictly ergodic.} and mean equicontinuous, with the rotation $R_\alpha$ as
    its maximal equicontinuous factor and the projection to the first coordinate
    as the factor map.
  \item The unique $\varphi$-invariant measure $\mu$ is the projection of the
    Lebesgue measure $\lambda$ on $\kreis$ onto some measurable graph, that is,
    it is of the form $\mu=(\Id_{\kreis}\times \gamma)_* \lambda$, where
    $\gamma:\kreis\to\kreis$ is measurable.  \listend
\end{thm}

Note that since the projection to the first coordinate is the factor map to the
MEF, all fibres are circles. In particular, there exist no injectivity points. A
genericity statement similar to that in
\cite{DownarowiczGlasner2015IsomorphicExtensionsAndMeanEquicontinuity} can also
be obtained (see Remark~\ref{r.genericity}(a)), but we will not focus on this
issue.\smallskip

The price we have to pay for the smoothness of the examples is that of a more
restricted setting. While the construction
in \cite{DownarowiczGlasner2015IsomorphicExtensionsAndMeanEquicontinuity} allows
to choose an arbitrary strictly ergodic systems as factor, our examples are
always extensions of irrational rotations with Liouvillean rotation
number. However, on the positive side, the Anosov-Katok construction allows to
exert additional control over the lifts of the resulting torus diffeomorphisms
to finite covering spaces, and also of all iterates. We can thus ensure that all
these mappings are uniquely ergodic. This entails that the finite lifts do not
have additional dynamical eigenvalues, so that their discrete spectrum coincides
with that of the original system and there has to be a continuous part of the
dynamical spectrum. A classical result of Katok and Stepin on cyclic
approximations \cite{KatokStepin1967PeriodicApproximations}, combined with
further modifications of the construction, allows to ensure that this new part
of the spectrum is singular continuous. Altogether, we obtain the following.

\begin{thm} \label{t.m:1-topomorphic_extensions}
  For any $m\in\N$, there exist $\cC^\infty$-diffeomorphisms $\varphi$ of the
  two-torus with the following properties.  \alphlist
  \item $\varphi$ is a skew product over some irrational rotation
    $R_\alpha:\kreis\to\kreis, x\mapsto x+\alpha\bmod 1$.
  \item $\varphi$ is totally strictly ergodic and a measure-theoretic $m$ to $1$
    extension of the irrational rotation $R_\alpha$, which is the maximal
    equicontinuous factor of the system.
  \item The unique $\varphi$-invariant measure $\mu$ is the projection of the
    Lebesgue measure $\lambda$ on $\kreis$ onto some $m$-valued measurable
    graph, that is, it is of the form $\mu=\sum_{j=1}^m (\Id_{\kreis}\times
    \gamma_j)_* \lambda$, where $\gamma_j:\kreis\to\kreis$ are measurable
    functions for $j=1\ld m$ and $\gamma_i(x)\neq \gamma_j(x)$ $\lambda$-almost
    surely for all $i\neq j$.\footnote {We call such systems
      \textit{$m$:$1$-topomorphic extensions (of the MEF)} in analogy to the
      notion of \textit{topoisomorphic extensions}, for which the topological
      factor map to the MEF is measure-theoretically one-to-one and hence
      measure-theoretic isomorphism.}
  \item The dynamical spectrum of $\varphi$ is given by the (discrete) dynamical
    spectrum of $R_\alpha$ and a singular continuous component.
    \listend
\end{thm}

In the case $m=2$, these examples are measure-theoretically similar to
(generalised) Thue-Morse subshifts \cite{Keane1968GeneralisedMorse}, and also to
strictly ergodic irregular Toeplitz flows constructed by Iwanik and Lacroix in
\cite{IwanikLacroix1994NonRegularToeplitz}. In both cases, the systems are also
measure-theoretically finite-to-one extensions of their MEF, exhibit the same
discrete spectrum as the MEF and equally show an additional singular continuous
part of the spectrum. However, the topological structure of these examples is
quite different, since Toeplitz flows always have a residual set of injectivity
points for the projection to the MEF, whereas almost all fibres over the MEF of
the generalised Thue-Morse subshift contain exactly two points.\medskip

\noindent
    {\em Structure of the article:}\ In Section~\ref{Preliminaries}, we provide
    all the required preliminaries on topological dynamics, spectral theory,
    mean equicontinuity and the Anosov Katok method. In
    Section~\ref{GDeltaMeanEquicontinuity}, we show that mean equicontinuity is
    a $G_\delta$-property. This observation has been made already in
    \cite{DownarowiczGlasner2015IsomorphicExtensionsAndMeanEquicontinuity}, but
    in order to simplify the Anosov Katok construction carried out in
    Section~\ref{MainConstruction} we provide a different
    $G_\delta$-characterisation of mean equicontinuity here. Finally, in
    Section~\ref{NewExamples}, we discuss how to modify the construction in
    order to ensure that all iterates of all finite lifts will still be strictly
    ergodic, no new dynamical eigenvalues occur and the additional spectral
    component is singular continuous.

\section{Preliminaries} \label{Preliminaries}

\subsection{Topological and measure-preserving dynamics} \label{TopDyn} 

We refer to standard textbooks such as
\cite{auslander1988minimal,BrinStuck2002DynamicalSystems,Walters1982ErgodicTheory,katok/hasselblatt:1997}
for the following basic facts on topological dynamics and ergodic
theory. Throughout this article, a {\em topological dynamical system} ({\em
  tds}) is a pair $(X,\varphi)$, where $X$ is a compact metric space and
$\varphi$ is a homeomorphism of $X$. We say $\varphi$ (or $(X,\varphi)$) is {\em
  minimal} if there exists no non-empty $\varphi$-invariant compact subset of
$X$. Equivalently, $\varphi$ is minimal if for all $x\in X$ the $\varphi$-orbit
$\cO_\varphi(x)=\{\varphi^n(x)\mid n\in\Z\}$ of $x$ is dense in $X$. The tds
$(X,\varphi)$ is called {\em equicontinuous} if for any $\eps > 0$ there exists
$\delta > 0$ such that $d_X(x,y) < \delta$ implies
$d_X(\varphi^n(x),\varphi^n(y))<\varepsilon$ for all $n \in \N$.  In this case,
there is an equivalent metric on $X$ such that $\varphi$ becomes an isometry.
When $(X,\varphi)$ is both equicontinuous and minimal, then $X$ can be given the
structure of a compact abelian group with group operation $\oplus$ such that
$\varphi$ is just the rotation by some element from $X$, that is, there exists
$\alpha\in X$ such that $\varphi(x)=x\oplus\alpha$ for all $x\in X$. We write
$x\ominus y$ for $x\oplus (\ominus y)$ in this situation, where $\ominus y$ is the inverse of $y$.
Since minimal rotations on
compact abelian groups are always uniquely ergodic, the same holds for minimal
equicontinuous systems.

Another tds $(Y,\psi)$ is called a {\em factor} of $(X,\varphi)$ with {\em
  factor map} $\pi:X\to Y$ if $\pi$ is continuous and onto and satisfies
$\pi\circ\varphi=\psi\circ \pi$. If in addition $\pi$ is a homeomorphism, we say
$(X,\varphi)$ and $(Y,\psi)$ are {\em conjugate}. Note that both minimality and
equicontinuity are inherited by factors. Since factor maps are in general not
unique, we will sometimes also refer to the triple $(Y,\psi,\pi)$ as a factor in
order to specify the factor map. We call such a triple a {\em maximal
  equicontinuous factor (MEF)} of $(X,\varphi)$ if $(Y,\psi)$ is equicontinuous
and for any other equicontinuous factor $(Z,\rho,p)$ there exists a unique
factor map $q$ between $(Y,\psi)$ and $(Z,\rho)$ such that $p=q\circ \pi$. The
existence of a MEF is ensured by the following statement, which also addresses
the question of uniqueness.

\begin{thm}[{\cite[Theorem 9.1, p.\ 125]{auslander1988minimal}}]
	Every topological dynamical system $(X,\varphi)$ has a MEF
        $(Y,\psi,\pi)$. If $(\hat Y, \hat\psi,\hat \pi)$ is another MEF, then
        there is a unique conjugacy $h: (Y,\psi) \to (\hat Y,\hat\psi)$ such
        that $\hat\pi=h \circ \pi$.  In particular, the systems $(Y,\psi)$ and
        $(\hat Y,\hat\psi)$ are conjugate in this case.
\end{thm}

\begin{rem} \label{r.mef_uniqueness}\alphlist
  \item
  Despite the lack of uniqueness, we will often refer to a MEF $(Y,\psi)$ of
  $(X,\varphi)$ as {\em `the MEF'}, in particular in situations where we are
  only interested in conjugacy-invariant properties.
\item Note also that once we have fixed an equicontinuous tds $(Y,\psi)$ as the
  MEF of a minimal system $(X,\varphi)$, the corresponding factor map $\pi$ is
  unique modulo post-composition with a rotation on $Y$ (where we refer to the
  above-mentioned group structure of minimal equicontinuous systems). The reason
  is the fact that in this case, given two different factor maps
  $\pi_1,\pi_2:X\to Y$, the $Y$-valued function $\pi_1\ominus\pi_2$ is
  continuous and $\varphi$-invariant, and therefore constant by minimality.
  \listend
\end{rem}

A {\em measure-preserving dynamical system (mpds)} is a quadruple
$(X,\cA,\mu,\varphi)$ consisting of a probability space $(X,\cA,\mu)$ and a
measurable transformation $\varphi:X\to X$ that preserves the measure $\mu$,
that is, $\varphi_*\mu=\mu$, where $\varphi_*\mu(A)=\mu(\varphi^{-1}(A))$. An
mpds is {\em ergodic} if every $\varphi$-invariant set $A\in\cA$ has measure $0$
or $1$. This is equivalent to the validity of the assertion of the Birkhoff
Ergodic Theorem: for any $f\in L^1(\mu)$, there holds
\begin{equation} \label{e.ergodic_averages}
\nLim \ntel \inergsum f\circ \varphi^i(x) \ = \ \int_X f\ d\mu
\end{equation}
for $\mu$-almost every $x\in X$. Given two mpds $(X,\cA,\mu,\varphi)$ and
$(Y,\cB,\nu,\psi)$, we call a measurable map $h:X\to Y$ a {\em measure-theoretic
  isomorphism} if there exist sets $A\in\cA,\ B\in\cB$ such that
$\mu(A)=\nu(B)=1$, $h:A\to B$ is a bi-measurable bijection, $h_*\mu=\nu$ and
$\varphi\circ h=\psi\circ h$ on $A$.

The mpds we consider will mostly be topological, that is, $X$ will be a compact
metric space, $\cA=\cB(X)$ the Borel $\sigma$-algebra on $X$, $\mu$ a Borel
measure and $\varphi$ a homeomorphism. In particular, this means that $\varphi$
is a bi-measurable bijection. For any tds $(X,\varphi)$, the existence of at
least one $\varphi$-invariant probability measure is ensured by the
Krylov-Bogolyubov Theorem. If there exists exactly one invariant measure --
which is necessarily ergodic in this case -- we call a tds {\em uniquely
  ergodic}. In this case, the Uniform Ergodic Theorem states that the
convergence of the ergodic averages in (\ref{e.ergodic_averages}) is uniform for
any continuous function $f$ on $X$. Actually, the same holds if $\varphi$ admits
multiple invariant measures, but the integral of the function $f$ is the same
with respect to all of them. This a more or less direct consequence of the
Krylov-Bogolyubov procedure and can be extended, to families of continuous
functions that are compact in the uniform topology, in the following way.

\begin{thm}[Simultaneous Uniform Ergodic Theorem]
	\label{thm:simergothm}
	Suppose that $(X,\varphi)$ is uniquely ergodic with invariant measure
        $\mu$.  For any compact family $\mathcal{F} \subseteq \cC(X,[0,1])$ of
        continuous functions and $x \in X$, the simultaneous ergodic averages
	\[ A_n: X \times \cF \longrightarrow \R\quad , \quad  (x,f)
        \longmapsto\frac{1}{n}\summe i0{n-1} f\kl \varphi^i(\cdot) \kr \]
        converge uniformly to the function $(x,f) \mapsto \int f \dd\mu$.
\end{thm}
We omit the proof, which is a straightforward adaptation of the standard
argument for the Uniform Ergodic Theorem (see e.g.\ \cite[Theorem
  6.19]{Walters1982ErgodicTheory}).

\subsection{Spectral theory of dynamical systems} \label{DynamicalSpectrum}

Given an mpds $(X,\cA,\mu,\varphi)$, the associated {\em Koopman operator} is
given by
\[
U_\varphi : L^2_\mu(X)\to L^2_\mu(X) \quad , \quad f\mapsto f\circ\varphi \ .
\]
Since $\mu$ is $\varphi$-invariant, $U_\varphi$ is a unitary operator, so that
$\sigma(U_\varphi)\ssq\mathbb{S}^1$. It is well-known that spectral properties
of $U_\varphi$ are closely related to dynamical properties of the system. For
instance, ergodicity of $\varphi$ is equivalent to the simplicity of $1$ as an
eigenvalue, and weak mixing of the system is equivalent the absence of any
further eigenvalues.

For any continuous function $f:\sigma(U_\varphi)\to\C$, the continuous
functional calculus yields the existence of a bounded linear operator
$f(U_\varphi)$ on $L^2_\mu(X)$ such that the mapping
\[
\cC(\sigma(U_\varphi),\C)\to \cC^*(U_\varphi)=\{f(U_\varphi)\mid
f\in\cC(\sigma(U_\varphi),\C)\} \quad , \quad f\mapsto f(U_\varphi)
\]
is an isomorphism of $\cC^*$-algebras. Given $g\in L^2_\mu(X)$, this further
allows to define a bounded linear functional
\[
\ell_g : \cC(\sigma(U_\varphi),\C)\to \C \quad , \quad f\mapsto \langle
f(U_\varphi)g,g\rangle
\]
and thus, by virtue of the Riesz Representation Theorem, a Borel measure $\mu_g$
on $\sigma(U_\varphi)\ssq \mathbb{S}^1$ such that
\[
\langle f(U_\varphi)g,g\rangle \ = \ \int_{\sigma(U_\varphi)} f\ d\mu_g \ .
\]
The measure $\mu_g$ is called the {\em spectral measure associated to
  $g$}. Moreover, there exists an orthogonal decomposition
\[
L^2_\mu(X) \ = \ L^2_\mu(X)_\mathrm{pp} \oplus L^2_\mu(X)_\mathrm{sc} \oplus
L^2_\mu(X)_\mathrm{ac} \ 
\]
of $L^2_\mu(X)$, where
\begin{eqnarray*}
	L^2_\mu(X)_\mathrm{pp} & = & \{ g\in L^2_\mu(X) \mid \mu_g \textrm{ is pure point}\} \ , \\
	L^2_\mu(X)_\mathrm{sc} & = & \{ g\in L^2_\mu(X) \mid \mu_g \textrm{ is singular continuous }\} \ , \\
	L^2_\mu(X)_\mathrm{ac} & = & \{ g\in L^2_\mu(X) \mid \mu_g \textrm{ is absolutely continuous}\} \ . \\
\end{eqnarray*}
In our setting the reference measure with respect to which the singularity and absolute continuity is defined is the 
Lebesgue measure on the circle.

The spectra $\sigma_\mathrm{pp}(U_\varphi)$, $\sigma_\mathrm{sc}(U_\varphi)$ and
$\sigma_{\mathrm{ac}}(U_\varphi)$ obtained from the restriction of $U_\varphi$
to these subspaces are called the {\em discrete/pure point}, {\em singular
  continuous} and {\em absolutely continuous part}, or {\em component}, of the
dynamical spectrum of $U_\varphi$. Note that the different spectral parts need
not be disjoint.

In the case of purely discrete spectrum, it turns out that a system is uniquely
characterised, up to isomorphism, by the group of its dynamical eigenvalues.

\begin{thm}[Halmos--von Neumann, \cite{VonNeumann1932Operatorenmethode,HalmosVonNeumann1942OperatorMethodsII}] \label{t.halmos-von_neumann}
  An ergodic mpds $(X,\cA,\mu,\varphi)$ has purely discrete spectrum if and only
  if it is measure-theoretically isomorphic to a minimal rotation of a compact
  abelian group equipped with its Haar measure.  Moreover, two mpds with purely
  discrete spectrum are are isomorphic if and only if they have the same group
  of eigenvalues.
\end{thm}

In order to prove the existence of a singular continuous spectral component, we
will use a classical result from the theory of approximations by periodic
transformations presented in \cite{KatokStepin1967PeriodicApproximations}. An
mpds $(X,\cA,\mu,\varphi)$ admits {\em cyclic approximation by periodic
  transformations (capt) with speed $s:\N\to \R^+_0$} if there exists a sequence
$\nfolge{\varphi_n}$ of bijective bi-measurable transformations on $X$ and a
sequence \nfolge{\cP_n} of finite measurable partitions of $X$,
$\cP_n=\{P_{n,1}\ld P_{n,K_n}\}$, such that for all $n\in\N$
\begin{itemize}
\item[(P1)] $\varphi_n$ cyclically permutes the elements of $\cP_n$;
\item[(P2)] for each $A\in\cA$, there exist $A_n\in\sigma(\cP_n)$ such that
  $\mu(A\Delta A_n)\nKonv 0$;
\item[(P3)] $\sum_{i=1}^{K_n} \mu(\varphi(P_{n,i})\Delta \varphi_n(P_{n,i}))
  \leq s(K_n)$.
\end{itemize}
Here, $\sigma(\cP_n)$ denotes the $\sigma$-algebra generated by $\cP_n$.

\begin{thm}[{\cite[Corollary 3.1]{KatokStepin1967PeriodicApproximations}}]
  \label{t.periodic_approximations}
  If an mpds $(X,\cA,\mu,\varphi)$ admits cyclic approximation by periodic
  transformations (capt) with speed $s:\N\to \R^+_0$ and $\nLim ns(n)=0$, then
  $U_\varphi$ has no absolutely continuous spectral component, that is,
  $\sigma_{\mathrm{ac}}(U_\varphi)=\emptyset$.
\end{thm}

\subsection{Mean equicontinuity} \label{MeanEquicontinuity}

A tds $(X,\varphi)$ is called mean equicontinuous if for all
$\varepsilon > 0$ there is $\delta > 0$ such that $d_X(x,y)<\delta$ implies
\begin{equation}\label{e.mean_equicontinuity_besicovich}
	d_\sB(x,y)\ = \ \varlimsup_{n\to\infty} \frac{1}{n} \inergsum d(\varphi^i(x),\varphi^i(y)) \ <
\ \varepsilon
\end{equation}

It turns out that in the minimal case, mean equicontinuity implies unique
ergodicity and is moreover equivalent to a certain invertibility property of the
factor map onto the MEF.
\begin{thm}[{\cite[Thm.\ 2.1]{DownarowiczGlasner2015IsomorphicExtensionsAndMeanEquicontinuity}}]
\label{thm:downarowiczglasner}
Suppose $(X,\varphi)$ is minimal and $(Y,\psi,\pi)$ is a MEF. Denote by $\nu$
the unique invariant measure of $(Y,\psi)$.  Then the following are equivalent.
\begin{enumerate}
 \item
  $(X,\varphi)$ is mean equicontinuous.
 \item
  $(X,\psi)$ is uniquely ergodic with unique invariant measure $\mu$ and $\pi$
   is a measure-theoretic isomorphism between the mpds $(X,\cB(X),\mu,\varphi)$
   and $(Y,\cB(Y), \nu, \psi)$.\qedhere
\end{enumerate}
\end{thm}
Given a uniquely ergodic tds $(X,\varphi)$ and a factor $(Y,\psi,\pi)$ (which is
automatically uniquely ergodic as well), we say $(X,\varphi)$ is an isomorphic
extension of $(Y,\psi)$ if $\pi$ is a measure-theoretic isomorphism between the
two mpds $(X,\cB(X),\mu,\varphi)$ and
$(Y,\cB(Y),\nu,\psi)$, where $\mu$ and $\nu$ are the unique invariant measures
for $\varphi$ and $\psi$, respectively. Hence, condition (2) above can be
rephrased by saying that $(X,\varphi)$ is an isomorphic extension of $(Y,\psi)$.

Note that the above statement implies, in particular, that the dynamical
spectrum of mean equicontinuous systems coincides with that of the MEF and is
therefore purely discrete.
\medskip

The mapping $d_\sB:X\times X\to\R^+_0$ defined
in~(\ref{e.mean_equicontinuity_besicovich}) is always a pseudo-metric on $X$. It
is called the {\em Besicovitch pseudo-metric.} For mean equicontinuous systems,
it provides a way to directly define a MEF of the system.

\begin{prop}[\cite{DownarowiczGlasner2015IsomorphicExtensionsAndMeanEquicontinuity}]
	\label{thm:besicovitchdistancezero} 
	Suppose $(X,\varphi)$ is mean equicontinuous. Define an equivalence
        relation on $X$ by
	 \[ x \sim y \quad \Leftrightarrow \quad  d_\sB(x,y) = 0 \ . \]
	Then the quotient system $(X/\!\sim,\xi/\!\sim)$ together with the
        canonical projection as a factor map is a MEF.
\end{prop}
The proof of this fact in
\cite{DownarowiczGlasner2015IsomorphicExtensionsAndMeanEquicontinuity} is
implicit -- it is contained in the proof of \cite[Theorem
  2.1]{DownarowiczGlasner2015IsomorphicExtensionsAndMeanEquicontinuity}.

\subsection{The Anosov-Katok method} \label{AnosovKatok}

The Anosov-Katok method is arguably one of the best-known and most widely used
constructions in smooth dynamics and allows to obtain a broad scope of examples
with particular combinations of dynamical properties. Although many readers will
already be familiar with the general method, we provide a brief introduction in
order to fix notation and comment on some specific issues that will be relevant
in our context. The construction of mean equicontinuous diffeomorphism of the
two-torus will then be carried out in Section~\ref{MainConstruction}, while the
modification required to obtain the finite-to-one extensions in
Theorem~\ref{t.m:1-topomorphic_extensions} will be discussed in
Section~\ref{NewExamples}. 

We restrict to the case of tori $\T^d=\R^d/\Z^d$ and denote by $\homeo\kl\T^d\kr$
the space of homeomorphisms of the $d$-dimensional torus, by $\cC^k\kl\T^d\kr$ the
space of $k$-times differentiable torus endomorphisms (including the cases
$k=\infty$ and $k=\omega$, were the later stands for `real-analytic') and let
\[
\diffeo^k\kl\T^d\kr \ = \ \left\{ \varphi\in\homeo(\T^d) \mid \varphi,\varphi^{-1}\in\cC^k\kl\T^d\kr\right\} \ .
\]
We will identify $\diffeo^0\kl\T^d\kr$ and $\homeo\kl\T^d\kr$. Further, we
denote the supremum metric on $\cC^0\kl\T^d\kr$ by $d_{\sup}$ and let
\[
d_k(\varphi,\psi) \ =
\ \max_{i=0}^k\max\left\{d_{\sup}\kl\varphi^{(i)},\psi^{(i)}\kr,d_{\sup}\kl
\kl\varphi^{-1}\kr^{(i)}, \kl\psi^{-1}\kr^{(i)}\kr \right\} \ . 
\]
be the standard metric on the space of torus diffeomorphisms.  By
$B^k_\eps(\psi)$, we denote the $\eps$-ball around $\psi$ in
$\diffeo^k(\T^d)$.\medskip

Our aim is to recursively construct a sequence \nfolge{\varphi_n} of torus
diffeomorphisms according to the following scheme.
\begin{itemize}
\item Each $\varphi_n$ will be of the form $H_n\circ R_{\rho_n}\circ
  H_n^{-1}$, where $R_{\rho} :\T^d\to\T^d,\ x\mapsto x+\rho$ denotes the
  rotation with rotation vector $\rho\in\T^d$ and $H_n\in\diffeo^{\infty}\kl
  \T^d\kr$.
\item The $H_n$ will be of the form $H_n=h_1\circ \ldots\circ h_n$, where each
  $h_n\in\diffeo^{\infty}\kl\T^d\kr$ commutes with the rotation
  $R_{\rho_{n-1}}$, that is, $h_n\circ R_{\rho_{n-1}} =
  R_{\rho_{n-1}} \circ h_n$. Note that consequently we have that $H_n\circ
  R_{\rho_{n-1}}\circ H_n^{-1}=\varphi_{n-1}$. Hence, at this stage, we have
  introduced the new conjugating map $h_n$, but the system has not changed yet.
\item When going from step $n-1$ to $n$ in the construction, we choose $h_n$
  first and only pick the new rotation vector $\rho_n$ afterwards. Therefore,
  the continuity of the mapping $\rho\mapsto H_n\circ R_\rho\circ H_{n}^{-1}$
  (with respect to the metric $d_k$ for any $k\in\N$) allows to control the
  difference between $\varphi_{n-1}$ and $\varphi_n$ in the respective metric.
\item As a consequence, we can ensure that the resulting sequence
  \nfolge{\varphi_n} is Cauchy in $\diffeo^k\kl\T^d\kr$ for any $k\in\N$, simply
  by recursively choosing $\rho_n$ sufficiently close to $\rho_{n-1}$ with
  respect to $d_n$ in the $n$-th step of the construction. This ensures that the
  $\varphi_n$ converge to some limit $\varphi\in\diffeo^{\infty}\kl\T^d\kr$.
\end{itemize}
So far, the above items explain how to ensure the convergence of the constructed
sequence \nfolge{\varphi_n}, but they do not yet specify how to obtain any
particular dynamical properties. It turns out, however, that the method is
tailor-made to realise any $G_\delta$-properties in the space
$\homeo(\T^d)$. Suppose we want to ensure that our limit diffeomorphism
$\varphi$ belongs to a set of torus homeomorphism $A\ssq \homeo\kl\T^d\kr$ which
is $G_\delta$, that is, it is of the form $A=\ncap U_n$ with
$U_n\ssq\homeo\kl\T^d\kr$ open. As discussed below, both minimal and uniquely
ergodic torus homeomorphisms can be characterised in this way. We then proceed
as follows.
\begin{enumerate}[i)]
\item By choosing $h_n$ and $\rho_n$ accordingly, we ensure that $\varphi_n\in
  U_n$. How this is done exactly depends on the property that defines $A$. This
  is actually the crucial step in the construction, and we will provide details
  further below.
\item \label{item:Gdelta} Since $U_n$ is open, there exists $\eta_n>0$ such that
  $\overline{B^0_{\eta_n}(\varphi_n)}\ssq U_n$. By ensuring that the distance
  between $\varphi_j$ and $\varphi_{j+1}$ is small and decays sufficiently fast
  for all $j\geq n$, this yields $\varphi=\lim_{j\to\infty}
  \varphi_j\in\overline{B^0_{\eta_n}(\varphi_n)}\ssq U_n$ as well. Since this
  works for all $n\in\N$, we obtain $\varphi\in A$.
\item \label{item:automaticUn} In order to ensure $\varphi_n\in U_n$, it will often be convenient not to go
  from $\varphi_{n-1}$ to $\varphi_n$ directly, but to pass through some
  intermediate map $\hat\varphi_{n-1}$ instead. This is, for example, useful to
  ensure that the limit system is minimal and/or uniquely ergodic. For instance,
  we may first choose a totally irrational rotation number $\hat\rho_{n-1}$
  and define $\hat\varphi_{n-1}=H_n\circ R_{\hat\rho_{n-1}}\circ
  H_n^{-1}$. Then $\hat\varphi_{n-1}$ is minimal and uniquely ergodic, since
  this is true for the irrational rotation $R_{\hat\rho_{n-1}}$. If $U_n$ is
  defined in such a way that it contains all minimal/uniquely ergodic torus
  homeomorphisms, then $\hat\varphi_{n-1}\in U_n$ is automatic. If $\varphi_n$
  is then chosen sufficiently close to $\hat\varphi_{n-1}$ (by choosing
  $\rho_n$ close to $\hat\rho_{n-1}$), we obtain $\varphi_n\in U_n$ as
  required. Further, if both $\hat\rho_{n-1}$ and $\rho_n$ are close enough
  to $\rho_{n-1}$, then $\varphi_n$ will also be close to $\varphi_{n-1}$ in
  $\diffeo^k(\T^d)$.
\item In Section~\ref{MainConstruction}, we will actually use a further
  modification of the above scheme and pass through an additional third map
  $\tilde\varphi_{n-1}=H_n\circ R_{\tilde\rho_{n-1}}\circ H_n^{-1}$, where
  $\tilde\rho_{n-1}$ is irrational, but not totally irrational (so the entries
  of the rotation vector are rationally related).
\end{enumerate}

The following works in high generality for any compact metric space $X$.
Using a very similar argument as made in \ref{item:Gdelta}) we obtain
\begin{prop}
	\label{thm:strictergo} 
	Let $A = \bigcap_{n\in\N} U_n \subseteq \Homeo(X)$ where the $U_n$ are
        open.  There exists a sequence $\kl \eta^A_n \kr_{n \in \N}$ of
        functions \[ \eta_n^A: \Homeo(X)^n \longrightarrow [0,2^{-n}] \] such
        that if a sequence $\nfolge{\hat\varphi_n}$ of torus homeomorphisms
        satisfies \[ d_0\kl \hat\varphi_n,\hat\varphi_{n+1} \kr < {\eta^A_n\kl
          \hat\varphi_1,\ldots,\hat\varphi_n \kr}\] and $\hat\varphi_n \in U_n$
        for all $n\in\N$, then it converges (by completeness) with limit
        \mbox{$\varphi=\nLim \hat\varphi_n \in A$} .
\end{prop}
\begin{proof}[Sketch of proof]
	If we set $\eta_n(\varphi_1,\ldots,\varphi_n)< 2^{-n}$ the sequence will be Cauchy and converge to $\varphi \in Z$.
	Recursively ensure that 
	\[ \eta_n(\varphi_1,\ldots,\varphi_n) < \min_{k=1}^{n-1}\kl \eta_k(\varphi_1,\ldots,\varphi_k) - d(\varphi_k,\varphi_n) \kr \,.\]
	This way we will at each step respect the previously imposed conditions.
	Then finally pick $\eta > 0$ such that $\overline{B_\eta(\varphi_n)}\subseteq U_n$
	and impose the new condition $\eta_n(\varphi_1,\ldots,\varphi_n) < \eta$.
	In total $\varphi_m \in B_\eta(\varphi_n) \subseteq \overline{B_\eta(\varphi_n)}$ for all $m > n$.
	So $\varphi \in \overline{B_\eta(\varphi_n)} \subseteq U_n$ for all $n \in \N$
	and thus $\varphi \in \bigcap_{n\in\N}U_n = A$.
\end{proof}

\subsection{$G_\delta$-characterisation of strict ergodicity}

It is well-known that the set $\homeo^\mathrm{se}(X)$ of strictly ergodic
homeomorphisms of $X$ is $G_\delta$ in $\homeo(X)$.  For the convenience of the
reader, we include a short proof of this folklore result.\smallskip

We first show that minimality is $G_\delta$.  Let $\varepsilon > 0$. A set $A
\subseteq X$ is called \textit{$\varepsilon$-dense} if 
$B_\varepsilon(A) = X$.  Observe that the mapping $\xi$ is minimal if and only
if for any $\varepsilon > 0$ there is $M \in \N$ such that $\ml x, \varphi(x),
\ldots, \varphi^{M}(x) \mr$ is $\varepsilon$-dense for any $x \in X$.
Furthermore, the $k$-th iterate of a homeomorphism
depends continuously on that homeomorphism.
Therefore
\[
U^{\min}_{M,\varepsilon} = \ml \psi \in \Homeo(X,X) \mm \forall x \in X: \ml x,
\psi(x), \ldots, \psi^M(x) \mr \text{ is }\varepsilon\text{-dense} \mr
\]
is open in the supremum norm.  If $\homeo^{\min}(X)$  denotes the set of all
minimal homeomorphisms of $X$, the above yields
\[ \homeo^{\min}(X) = \bigcap_{\varepsilon \in \Q^+} \bigcup_{M\in\N} U^{\min}_{M,\varepsilon} \,. \]
This means in particular that $\homeo^{\min}(X)$ is $G_\delta$.\smallskip

Now, we turn to unique ergodicity.  Fix a dense set $\ml s_n \mm n \in \N \mr
\subset \cC(X,\R)$.  Given $n \in \N$ and a continuous map $\xi:X\to X$, we can
assign to any $g \in \cC(X,\R)$ its $n$-step ergodic average
\[ A_n^\xi g = \frac{1}{n}\sum_{i=0}^{n-1} g\circ \xi^i \, .\]
It is well known (e.g.\ \cite[Thm.\ 4.10, p.\ 105]{Einsiedler2010-ln}) that
$\xi$ is uniquely ergodic if and only if for every $k \in \N$ the sequence of
functions $(A_n^\xi s_k)_{n\in\N}$ converges pointwise to a constant.  For $g
\in \cC(X,\R)$ we denote by $V(g)$ its variation over $X$, that is
\[
	V(g) := \sup_{x\in X} g(x)- \inf_{y \in X} g(y)\  .
\]
Given any $\psi$-invariant probability measure $\mu$, we have $\int A_n^\psi g
\dd\mu = \int g \dd\mu$.  This implies that if $V(A_n^\psi g) \xrightarrow{n\to\infty} 0$, then
$A_n^\psi g \xrightarrow{n\to\infty} \int g \dd\mu$.

As $A_k^\xi g$ depends continuously on $\xi$ and $V(f)$ depends continously on
$f$, we see that
\[
U^{\mathrm{ue}}_{n,K,\beta} := \ml {\psi} \in \Homeo(X,X) \mm V(A_K^{\psi} s_n)
< \beta \mr
\]
is open in the supremum metric.  In particular, the set of those transformations
for which the ergodic average of $s_n$ eventually stabilises at a variation
below $\beta$, given by
\[
U^{\mathrm{ue}}_{n,\beta} = \bigcup_{K \in \N}
\bigcap_{k > K} U^{\mathrm{ue}}_{n,k,\beta}
\]
is $G_\delta$.  We notice that $\psi \in \bigcap_{k\in\N} \bigcap_{\beta \in
  \Q^+} U^{\mathrm{ue}}_{k,\beta}$ if and only if $A_n^\psi s_n$ converges to
$\int s_n \dd\mu$ for any $n \in \N$ if and only if $\psi$ is uniquely ergodic.
This yields that the set $\homeo^\mathrm{ue}(X)$ of uniquely ergodic
homeomorphisms of $X$ is $G_\delta$.
\smallskip

Altogether, this shows that
$\homeo^\mathrm{se}(X)=\homeo^\mathrm{min}(X)\cap\homeo^\mathrm{ue}(X)$ is a
$G_\delta$-set as claimed.

\section{$G_\delta$-characterisation of mean equicontinuity for skew products}
\label{GDeltaMeanEquicontinuity}

As discussed in the previous section, in order to construct mean equicontinuous
systems via the Anosov Katok method, it is instrumental to have an explicit
$G_\delta$-characterisation of mean equicontinuity available. In principle, such
a characterisation is already contained in
\cite{DownarowiczGlasner2015IsomorphicExtensionsAndMeanEquicontinuity}. However,
the latter uses the fact that mean equicontinuity is equivalent to the existence
of a unique self-joining on the product space $X\times X$ over the MEF as a
common factor. Since we want to avoid working in the product space, as this
would rather complicate the construction in the next section, we provide an
alternative characterisation here. As in
\cite{DownarowiczGlasner2015IsomorphicExtensionsAndMeanEquicontinuity}, we make
use of the fact that we are in a skew product setting and the factor map is
given {\em a priori} (by the projection to the first coordinate). We formulate
the statement in abstract terms, as it might be useful in other situations as
well.

\begin{prop}
  \label{prop:characterizingme} Let $(X,\varphi)$ be a tds and
  $(Y,\psi,\pi)$ an equicontinuous factor. Then $(Y,\psi,\pi)$ is a MEF of
  $(X,\varphi)$ and $(X,\varphi)$ is an isomorphic extension of $(Y,\psi)$ if
  and only if for all $\varepsilon > 0$ there exists some $K \in \N$ such that,
  for all $x,y\in X$, we have \begin{align} \pi(x)=\pi(y) \quad \Longrightarrow
    \quad \frac 1{K} \sum_{i=0}^{K} d_X(\varphi^i(x),\varphi^i(y)) < \varepsilon
    \,. \label{eq:finitetimeaverages} \end{align}
\end{prop}

\begin{rem} \label{r.mean_equicontinuity}
  \alphlist
  \item Denote by $\homeo^\mathrm{eq}(Y)$ the space of equicontinuous homeomorphisms
  of $Y$. Consider the space
  \[
  \cE(\pi)\ = \ \{\varphi\in\homeo(X)\mid \exists
  \psi\in\homeo^\mathrm{eq}(Y): \pi\circ\varphi=\psi\circ \pi\}
  \]
  with the subspace
  \[
    \cE^{\mathrm{iso}}(\pi) \ =
  \ \left\{\varphi\in\cE(\pi)\left| \begin{split} \exists & \psi\in\homeo(Y):
  (Y,\psi,\pi) \textrm{ is the MEF of } (X,\varphi)\\ & \quad \textrm{ and }
  (X,\varphi) \textrm{ is its isomorphic extension} \end{split}\right.\right\}
   \ .
  \]
  Let
  \[
  U_n^{\mathrm{iso}}(\pi) \ = \ \{ \varphi\in\cE(\pi)\mid \exists
  K\in\N: (\ref{eq:finitetimeaverages}) \textrm{ holds with } \eps=1/n \} \ .
  \]
  Then, by the above statement, we have $\cE^\mathrm{iso}(\pi)=\ncap
  U_n^{\mathrm{iso}}(\pi)$. As the sets $U_n^{\mathrm{iso}}(\pi)$ are open, this
  implies that $\cE^{\mathrm{iso}}(\pi)$ is a $G_\delta$-set in $\cE(\pi)$.

  Note here that property (\ref{eq:finitetimeaverages}) only depends on the
  factor map $\pi$, but not on the map $\psi$ acting on the factor space.
\item
  A similar characterisation could be given with a fixed factor system
  $(Y,\psi)$ on the base. However, in the context of the Anosov Katok
  construction, where the base systems of the approximating diffeomorphisms will
  be circle rotations with varying rotation numbers, the independence of $\psi$
  in the above characterisation is crucial.
\item According to Proposition~\ref{thm:strictergo}, there are
  mappings
  \[ \eta_n^\mathrm{me}:
		\Homeo(\T^d)^n \longrightarrow [0,2^{-n}] \] such that if a sequence
          $\nfolge{\tilde\varphi_n}$ in $\cE(\pi)$ satisfies $\tilde\varphi_n\in
          U_n^\mathrm{iso}(\pi)$ and \[ d_0\kl
          \tilde\varphi_n,\tilde\varphi_{n+1} \kr < {\eta^\mathrm{me}_n\kl
            \tilde\varphi_1,\ldots,\tilde\varphi_n \kr}\] for all $n\in\N$, then
          its limit \mbox{$\varphi=\nLim \tilde\varphi_n$} exists and belongs to
          $\cE^\mathrm{iso}(\pi)$.
          \listend
\end{rem}
\begin{proof}[\bfseries{Proof of Proposition~\ref{prop:characterizingme}}]
First, assume that $\kl Y,\psi,\pi\kr$ is a MEF and $\kl X,\varphi\kr$ is an
isomorphic extension of $\kl Y,\psi\kr$.  Then, by Theorem~
\ref{thm:besicovitchdistancezero}, $\pi(x)=\pi(y)$ implies $d_\sB(x,y)=0$.  Let
\[
\cJ_\pi\kl X\kr\ =\ \{(x,y)\in X\times X\mid \pi(x)=\pi(y)\}\ .
\]
Then $d_\sB(x,y)$ is the ergodic average of the function $d_X$ for the action of
$\varphi\times\varphi$ on $X\times X$. Since these ergodic averages are
identically zero on the compact invariant set $\cJ_\pi\kl X\kr$, we have that
$\int_{\cJ_\pi\kl X\kr} d_X(x,y) \ d\gamma(x,y)=0$ for all
$\varphi\times\varphi$-invariant measures $\gamma$ on $\cJ_\pi\kl X\kr$ (in
fact, by \cite[Proposition
  2.5]{DownarowiczGlasner2015IsomorphicExtensionsAndMeanEquicontinuity}, there
is only one such measure when $(X,\varphi)$ is mean equicontinuous). The Uniform
Ergodic Theorem therefore implies that the functions
\[
a_K(x,y)\ = \ \frac 1{K} \sum_{i=0}^{K} d_X\kl\varphi^i(x),\varphi^i(y)\kr
\]
 uniformly converge to zero as $K\to\infty$. Hence, we have $a_K<\eps$ for
 sufficiently large $K\in\N$, which is just an equivalent reformulation of
 (\ref{eq:finitetimeaverages}).\smallskip

Conversely, suppose that for all $\eps>0$ there exists $K\in\N$ such that
(\ref{eq:finitetimeaverages}) holds. We assume without loss of generality that
$\psi$ is an isometry. Denote by $\Delta_Y$ and $\Delta_X$ the diagonals in the
respective product spaces $Y\times Y$ and $X\times X$. Note that thus
$\cJ_\pi\kl X\kr=\kl\pi\times\pi\kr^{-1}\kl\Delta_Y\kr$.

Now, fix $\eps>0$ and choose $K\in\N$ according to
(\ref{eq:finitetimeaverages}). This means that the function $a_K$ is strictly
smaller than $\eps$ on $\cJ_\pi\kl X\kr$. By compactness of $\cJ_\pi\kl X\kr$
and continuity of $a_K$, there exists $\delta_1>0$ such that $a_K<\eps$ on
$B_{\delta_1}\kl J_\pi\kl X\kr\kr$. Due to the continuity of $\pi\times\pi$,
there exists $\eta>0$ such that $A =\kl\pi\times\pi\kr^{-1}\kl
\overline{B_\eta\kl\Delta_Y\kr}\kr\ssq B_{\delta_1}\kl \cJ_\pi\kl
X\kr\kr$. Further, as $\pi$ is uniformly continuous, there exists $\delta>0$
such that $\pi\times\pi\kl B_\delta\kl \cJ_\pi(X)\kr\kr \ssq
B_\eta\kl\Delta_Y\kr$.

As $\psi$ is an isometry, the set $A$ is $\varphi\times\varphi$-invariant. Since
$d_\sB$ is equal to the ergodic average of $a_K$ for the action of
$(\varphi\times\varphi)^K$ and $a_K<\eps$ on $A$, this yields $d_\sB<\eps$ on
$A$. However, by the above choices, $d_X(x,y)<\delta$ implies
$d_Y(\pi(x),\pi(y))<\eta$ and therefore $(x,y)\in A$. Hence, we obtain that
$d_X(x,y)<\delta$ implies $d_\sB(x,y)<\eps$. As $\eps>0$ was arbitrary, this
means that $(X,\varphi)$ is mean equicontinuous.

Therefore Theorem~\ref{thm:downarowiczglasner} implies that $(X,\varphi)$ is an
isomorphic extension of its MEF, which we denote by $(\hat Y,\hat\psi,\hat
\pi)$. By definition, we know that $(Y,\psi)$ is a factor of the MEF, so the
fibres of $\hat\pi$ are contained in the fibres of $\pi$. However, since the
Besicovitch distance $d_\sB$ between two points is zero on each fibre of $\pi$,
and as this property characterises the fibres of the MEF due to
Proposition~\ref{thm:besicovitchdistancezero} (note here that the fibres do not
depend on the particular choice of the MEF), the fibres of $\pi$ are also
contained in the fibres of $\hat \pi$. Thus, $\pi$ and $\hat\pi$ have the same
fibres, which implies that $(Y,\psi,\pi)$ is also a MEF and $(X,\varphi)$ is an
isomorphic extension of $(Y,\psi)$.
\end{proof}

\section{Mean equicontinuous skew products on the torus:
  Proof of Theorem~\ref{t.meanequi_with_full_fibres}}
\label{MainConstruction}

We are going to construct a mean equicontinuous diffeomorphisms of the two-torus
which have skew product form
\begin{equation}\label{e.mainconst_skewproduct}
  \varphi:\torus\to\torus\quad , \quad (x,y)\mapsto (x+\alpha,\varphi_x(y)) \  
\end{equation}
and are such that the underlying irrational rotation
$R_\alpha:\kreis\to\kreis,\ x\mapsto x+\alpha$ is the MEF and the factor map is
given by the projection $\pi:\torus\to\kreis,\ (x,y)\mapsto x$ to the first
coordinate. In order to do so, we employ the Anosov Katok method as described in
Section \ref{AnosovKatok} and recursively define sequences of skew product
diffeomorphism \nfolge{\varphi_n}, \nfolge{\tilde\varphi_n} and
\nfolge{\hat\varphi_n} whose common limit $\varphi$ will satisfy the assertions
of Theorem~\ref{t.meanequi_with_full_fibres}. The general scheme of our
inductive construction will be as follows.
\begin{itemize}
\item The mappings $\varphi_n,\tilde\varphi_n$ and $\hat\varphi_n$ will be of
  the form 
		\[ \varphi_n=H_n\circ R_{\rho_n}\circ H_n^{-1}\,,\quad \tilde\varphi_n=H_{n+1}\circ R_{\tilde\rho_n}\circ H_{n+1}^{-1}
		\quad \text{ and }\quad \hat\varphi_n=H_{n+1}\circ
                R_{\hat\rho_n}\circ H_{n+1}^{-1} \,,\] where $\rho_n$ is
                rational, $\tilde\rho_n=\alpha\rho_n$ with $\alpha\in\R\smin\Q$
                and $\hat\rho_n$ is totally irrational. Further, for technical
                reasons, we require that
  \begin{equation}\label{e.relatively_prime}
  \rho_n \ = \ \left(\frac{p_n}{q_n},\frac{p_n'}{q_n}\right) \quad \textrm{ with }
  p_n,p_n',q_n\in\N \textrm{ relatively prime} \ .
  \end{equation}
\item The conjugating diffeomorphisms $H_n$ will be of the form $H_n=h_1\circ \ldots \circ h_n$,
  where $h_{n+1}$ always commutes with the rotation
  $R_{\rho_n}$. Moreover, all $h_n$ have skew product structure
  $h_n:(x,y)\mapsto (x,h_{n,x}(y))$, with fibre maps $h_{n,x}:\kreis\to\kreis$.
\item We choose the functions $\eta^\mathrm{se}_n$ and $\eta^\mathrm{me}_n$ and
  the sets $U_n^\mathrm{iso}(\pi)$ according to Proposition~\ref{thm:strictergo}
  and Remark~\ref{r.mean_equicontinuity}(c).
\item The approximating torus diffeomorphisms
  $\varphi_n,\tilde\varphi_n,\hat\varphi_n$ will be chosen such that for each
  $n\geq 2$ they satisfy
  \begin{eqnarray}
    d_n\kl\varphi_n,\tilde\varphi_{n}\kr & \leq & \drittel\label{e.mainconstr1}
    \min\left\{\eta^\mathrm{se}_{n-1}\kl\hat\varphi_1\ld
    \hat\varphi_{n-1}\kr,\eta^\mathrm{me}_{n-1}\kl\tilde\varphi_1\ld
    \tilde\varphi_{n-1}\kr\right\} \\ \label{e.mainconstr2}
    d_n\kl\tilde\varphi_n,\hat\varphi_{n}\kr & \leq & \drittel
    \min\left\{\eta^\mathrm{se}_{n-1}\kl\hat\varphi_1\ld
    \hat\varphi_{n-1}\kr,\eta^\mathrm{me}_{n}\kl\tilde\varphi_1\ld
    \tilde\varphi_{n}\kr\right\} \\ \label{e.mainconstr3}
    d_n\kl\hat\varphi_n,\varphi_{n+1}\kr & \leq & \drittel
    \min\left\{\eta^\mathrm{se}_{n}\kl\hat\varphi_1\ld
    \hat\varphi_{n}\kr,\eta^\mathrm{me}_{n}\kl\tilde\varphi_1\ld
    \tilde\varphi_{n}\kr\right\} \\ \tilde\varphi_n & \in &
    U^\mathrm{iso}_n(\pi) \ .\label{e.phin_in_Uniso}
  \end{eqnarray}
  Note that these conditions together imply that
  \begin{eqnarray}
    d_n\kl\tilde\varphi_n,\tilde\varphi_{n+1}\kr & \leq &
    \eta_n^\mathrm{me}\kl\tilde\varphi_1\ld
    \tilde\varphi_n\kr\\ d_n\kl\hat\varphi_n,\hat\varphi_{n+1}\kr & \leq &
    \eta_n^\mathrm{se}\kl\hat\varphi_1\ld \hat\varphi_n\kr
  \end{eqnarray}
  for all $n\in\N$, which together with (\ref{e.phin_in_Uniso}) means that the
  conditions of Proposition~\ref{thm:strictergo} and
  Remark~\ref{r.mean_equicontinuity}(c) are met, where
  Proposition~\ref{thm:strictergo} is applied to the sequence
  \nfolge{\hat\varphi_n} and Remark~\ref{r.mean_equicontinuity}(c) is applied to
  \nfolge{\tilde\varphi_n}. Note here that the diffeomorphisms $\hat\varphi_n$
  are all strictly ergodic, since they are conjugate to a totally irrational
  torus rotation. Consequently, the common limit $\varphi$ of all the sequences,
  whose existence is also guaranteed by
  (\ref{e.mainconstr1})--(\ref{e.mainconstr3}) (note that
  $\eta^\mathrm{me}_n,\eta^\mathrm{se}_n\leq 2^{-n}$), is both strictly ergodic
  and mean equicontinuous, with $\pi$ as the factor map to the MEF. The MEF is
  then given by $(\kreis,R_\alpha,\pi)$, where $\alpha=\nLim \pi(\rho_n)$.
  \item Note that conditions (\ref{e.mainconstr1})--(\ref{e.mainconstr3}) can
    always be ensured by choosing the rotation vectors
    $\rho_n,\tilde\rho_n,\hat\rho_n$ and $\rho_{n+1}$ sufficiently close to each
    other. The reason is the fact that we have $\varphi_n=H_{n+1}\circ
    R_{\rho_n}\circ H_{n+1}^{-1}$ due to the commutativity between $h_{n+1}$ and
    $R_{\rho_n}$ combined with the continuity of $\rho\mapsto H_{n+1}\circ
    R_\rho\circ H_{n+1}^{-1}$.  Therefore, the only issue that remains to be
    addressed is to ensure that the intermediate maps $\tilde\varphi_n$ are
    indeed contained in $U^\mathrm{iso}_n(\pi)$.
\end{itemize}

In order to start the induction, we let $H_1=h_1=\Id_{\torus}$ and choose
$\varphi_1$ to be an arbitrary rational rotation on \torus\ whose rotation
vector satisfies~(\ref{e.relatively_prime}). Note that the inductive assumptions
(\ref{e.mainconstr1})--(\ref{e.mainconstr3}) are all still empty at this point.

Now, suppose that $\varphi_1\ld \varphi_N$,
$\tilde\varphi_1\ld\tilde\varphi_{N-1}$ and $\hat\varphi_1\ld\hat\varphi_{N-1}$
have been constructed such that (\ref{e.mainconstr1})--(\ref{e.phin_in_Uniso})
hold for all $n=1\ld N-1$. We have
$\rho_n=\left(\frac{p_n}{q_n},\frac{p_n'}{q_n}\right)$, where
$p_n,p_n',q_n\in\N$ are relatively prime. The aim is to choose $h_{n+1}$ and
$\tilde\rho_n$ in such a way that $\tilde\varphi_n\in U^\mathrm{iso}_n(\pi)$,
that is, $\tilde\varphi_n$ satisfies (\ref{eq:finitetimeaverages}) with
$\eps=1/n$.

To that end, we note that orbits of $R_{\rho_n}$ move along closed curves of the
form
\[
L(\rho_n,t) \ = \ \left\{(xp_n/q_n,t+xp_n'/q_n)\mid x\in [0,q_n) \right\} \ ,
  \]
  which are parametrised by the functions
  \[
  \ell_{\rho_n,t}:\kreis\to L(t,\rho_n)\quad , \quad x\mapsto (xp_n,t+xp_n') \ .
  \]
We now dwell on this insight a bit further in order to see how we need to choose
$h_{n+1}$. First, observe that the mapping $\ell_{\rho_n,t}$ conjugates the
one-dimensional rotation $r_{1/q_n}$ on $\kreis$ and the restriction of
$R_{\rho_n}$ to $L(t,\rho_n)$, that is, $R_{\rho_n}\circ \ell_{\rho_n,t} =
\ell_{\rho_n,t}\circ r_{1/q_n}$. Consequently, the orbits of
$\varphi_n=H_{n+1}\circ R_{\rho_n}\circ H_{n+1}^{-1}$ move along the curves
$H_{n+1}(L(\rho_n,t))$, and $H_{n+1}\circ \ell_{\rho_n,t}$ provides a conjugacy
between the action of $\varphi_n$ on these curves and the rational rotation
$r_{1/q_n}$. Moreover, if we change the rotation vector $\rho_n$ to
$\tilde\rho_n=\alpha\rho_n$, where $\alpha$ is an irrational real number, then
the orbits of $\tilde\varphi_n$ still move along the same curves, but now the
action of $\tilde\varphi_n$ on these curves is conjugate to the irrational
rotation $r_{\alpha/q_n}$ on $\kreis$ (again with conjugacy $H_{n+1}\circ
\ell_{\rho_n,t}$).

Given two points $z,z'\in\torus$ with $\pi(z)=\pi(z')$ we may choose
$x,t_z,t_{z'}\in\kreis$ such that $z=H_{n+1}\kl\ell_{\rho_n,t_z}(x)\kr$ and
$z'=H_{n+1}\kl\ell_{\rho_n,t_{z'}}(x)\kr$. For the average distance of the
iterates of these two points along their orbits, we obtain
\begin{equation} \label{e.distance_averages} \begin{split}
  \ntel\inergsum d\kl\tilde\varphi^i_n(z),\tilde\varphi^i_n(z')\kr & =
  \ \ntel\inergsum
  d\kl\tilde\varphi^i_n\kl\ell_{\rho_n,t_z}(x)\kr,\tilde\varphi^i_n
  \kl\ell_{\rho_n,t_{z'}}(x)\kr\kr \\ & = \ \ntel\inergsum F_{n,t_z,t_z'}\circ
  r_{\alpha/q_n}^i(x) \ ,
  \end{split}
\end{equation}
where
\[
F_{n,t,t'} : \kreis\to\R^+ \quad , \quad x\mapsto d\kl
H_{n+1}\circ\ell_{\rho_n,t}(x),H_{n+1}\circ\ell_{\rho_n,t'}(x)\kr \ .
\]
By unique ergodicity of the irrational rotation $r_{\alpha/q_n}$, the averages
in (\ref{e.distance_averages}) converge uniformly to $\int_{\kreis}
F_{n,t,t'}(x)\ dx$. As the family $\{F_{n,t,t'}\mid t,t'\in\kreis\}$ is compact,
the Simultaneous Uniform Ergodic Theorem~\ref{thm:simergothm} implies that this
convergence is even uniform in the parameters $t,t'\in\kreis$. This means that
for any $\kappa>0$ there exists $K\in\N$ such that
\[
\left| \frac{1}{K}\sum_{i=0}^{K-1}
d\kl\tilde\varphi^i_n(z),\tilde\varphi^i_n(z')\kr - \int_{\kreis}
F_{n,t_z,t_{z'}}(x)\ ds\right| \ < \ \kappa
\]
holds for all $z,z'\in\kreis$ with $\pi(z)=\pi(z')$. Hence, in order to ensure
the validity of (\ref{e.phin_in_Uniso}), it suffices to let $\kappa=1/2n$ and to
choose $h_{n+1}$ in such a way that
\begin{equation} \label{e.integrals_estimate}
\int_{\kreis} F_{n,t,t'}(x)\ dx \ \leq \ \frac{1}{2n}
\end{equation}
holds for all $t,t'\in\kreis$.

Now, constructing such a mapping $h_{n+1}$ is not difficult, albeit somewhat
technical. We first define $h_{n+1}$ on the vertical strip $S=I\times\kreis$,
where $I=[0,1/q_n]$. We fix $\delta>0$ and choose some circle diffeomorphism
$g$, homotopic to the identity, such that $g\kl\kreis\smin B_\delta(1/2)\kr \ssq
B_\delta(0)$. For instance, $g$ could be the projective action of a diagonal
matrix {\tiny$\begin{pmatrix} \lambda & 0 \\ 0 & 1/\lambda \end{pmatrix}$} with
sufficiently large $\lambda>0$. Then we choose a smooth homotopy
$G:[0,1]\times\kreis\to\kreis$, $(x,y)\mapsto G_x(y)$ between $G_0=\Id_{\kreis}$
and $G_1=g$ such that $G_x=\Id_{\kreis}$ for all $x$ in some neighbourhood of
zero. We assume $\delta\in[0,1/4q_n]$, let $\hat I=[\delta,1/q_n-\delta]$ and
choose a smooth mapping $T:S\to S$ such that $T(x,y)=(x,T_x(y))$ where
$T_x(y)=y+\frac{x}{1/q_n-2\delta}$ for all $(x,y)\in \hat I$ and
$T_x=\Id_{\kreis}$ for all $x$ in a neighbourhood of zero. Thus, the image of a
horizontal line segment $\hat I\times\{y\}$ under $T$ `wraps' around the torus
exactly once in the vertical direction.\smallskip

Using these auxiliary mappings, we let
\begin{equation}
  h_{n+1} : S\to S \quad , \quad (x,y)\mapsto 
  \begin{cases}
    G_{x/\delta}\circ T(x,y) & 0\leq x \leq \delta \\ g\circ T(x,y) &
    \delta<x<1/q_n-\delta \\ G_{(1-x)/\delta}\circ T(x,y) & 1/q_n-\delta \leq x
    \leq 1/q_n
  \end{cases} \ . 
\end{equation}
Then $h_{n+1}$ is smooth on $S$ and coincides with the identity on a
neighbourhood of $\partial S$.

Now, we first focus on that segment of a curve $L(\rho_n,t)$ passing through
$S$, which is parametrised by the mapping $p_n^{-1}I\to \kreis,\ x\mapsto
\ell_{\rho_n,t}(x)$. Note that there are $p_n$ such pieces, which all differ by
an additive constant that is a multiple of $1/p_n$. On $p_n^{-1}\hat I$, this
function has constant slope $\frac{1}{1/q_n-2c}+\frac{p_n}{q_n}$. As a
consequence, it passes through the set $\hat I\times B_\delta(1/2)$ at most
twice over the interval $\hat S$ and we obtain that
\begin{equation}\label{e.ell_curves_estimate_1}
	\Leb_{\kreis}\left(\left\{ x\in p_n^{-1}I\mid T_x\circ \ell_{\rho_n,t}(x)\in
	B_\delta(1/2)\right\}\right) \ \leq \ 2\delta+\frac{\delta}{p_nq_n} \ .
\end{equation}
Since $g$ maps the complement of the interval $B_\delta(1/2)$ into the interval
$B_\delta(0)$ and (\ref{e.ell_curves_estimate_1}) holds for all $t\in\kreis$,
the images of respective segments of two different curves $L(q_n,t),L(q_n,t')$
under $T$ will be $2\delta$-close to each other most of the time. More
precisely, we obtain
\begin{equation}\label{e.ell_curves_estimate_2}
	\Leb_{\kreis}\left(\left\{ x\in p_n^{-1}I\mid \left| h_{n+1,s}\circ
\ell_{\rho_n,t}(x)- h_{n+t,s}\circ \ell_{\rho_n,t'}(x)\right| \geq
2\delta\right\}\right) \ \leq \ 4\delta+\frac{2\delta}{p_nq_n} \ .
\end{equation}
So far, we have only defined $h_{n+1}$ on $S$ and only considered the
restriction of the curves $\ell_{\rho_n,t}$ to the interval $p_n^{-1}I$, which only
parametrises the $1/p_nq_n$-th part of the whole curves $L(\rho_n,t)$. However,
if we extend the definition of $h_{n+1}$ by commutativity to all of \torus, i.e.~by setting
$h_{n+1|R_{\rho_n}^k(S)}=R_{\rho_n}^k\circ h_{n+1|S}\circ R_{\rho_n}^{-k}$,
$k=1\ld q_n-1$, then the behaviour of all $p_nq_n$ segments of pairs of curves
$h_{n+1}(L(q_n,t))$ and $h_{n+1}(L(q_n,t'))$ will be the same -- we are simply
looking at a rotated version of the same situation. Therefore, we obtain the
estimate
  \begin{equation}\label{e.ell_curves_estimate_3}
\Leb_{\kreis}\left(\left\{ x\in \kreis\mid \left| h_{n+1,x}\circ
\ell_{\rho_n,t}(x)- h_{n+1,x}\circ \ell_{\rho_n,t'}(x)\right| \geq
2\delta\right\}\right) \ \leq \ 6\delta p_nq_n\ .
\end{equation}
  Since $H_n$ is uniformly continuous, we may choose $\delta$ in such a way that
  $d(x,y)<2\delta$ implies $d(H_n(x),H_n(y))<\frac{1}{4n}$.
  Then~(\ref{e.ell_curves_estimate_3}) implies
    \begin{equation}\label{e.ell_curves_estimate_4}
\Leb_{\kreis}\left(\left\{ s\in \kreis\mid \left| H_{n+1,s}\circ
\ell_{\rho_n,t}(s)- H_{n+t,s}\circ \ell_{\rho_n,t'}(s)\right| \geq
1/4n\right\}\right) \ \leq \ 6\delta p_nq_n \ .
    \end{equation}
 When $\delta$ is sufficiently small (say $\delta<1/24np_nq_n$), this finally
 yields~(\ref{e.integrals_estimate}).

 In order to complete the induction step, we now choose
 $\tilde\rho_n=\alpha\rho_n$, where $\alpha\in\R\smin\Q$ is sufficiently close
 to $1$ such that (\ref{e.mainconstr1}) holds. After that, we can take
 $\hat\rho_n$ to be any totally irrational rotation vector, close enough to
 $\tilde\rho_n$ to ensure~(\ref{e.mainconstr2}), and finally choose a new
 rational rotation vector $\rho_{n+1}$ that is close enough to $\hat\rho_n$ to
 ensure~(\ref{e.mainconstr3}) and satisfies (\ref{e.relatively_prime}) holds.
 This completes the inductive construction and therefore the proof of
 Theorem~\ref{t.meanequi_with_full_fibres}.

 \begin{rem}\label{r.genericity}
   \alphlist
   \item
     The above construction starts with an arbitrary rational rotation
     $R_{\rho_1}$ and allows to ensure that the resulting limit diffeomorphism
     $\varphi$ is arbitrarily close to $R_{\rho_1}$. Together with the
     $G_\delta$-property of mean equicontinuity and strict ergodicity, this
     implies that the set of skew product diffeomorphisms which satisfy the
     assertions of Theorem~\ref{t.meanequi_with_full_fibres} form a residual
     subset of the space
     \[
     \overline{\mathrm{Cob}}(\torus,\pi) \ = \ \left\{ H\circ R_{\rho}\circ
     H^{-1} \mid \rho\in\torus,\ H \in \diffeo^k(\torus),\ \pi\circ
     H=\pi\right\} \ , 
     \]
     where $k\in\N_0\cup \{\infty\}$ is arbitrary. The analogous observation has
     been made in
     \cite{DownarowiczGlasner2015IsomorphicExtensionsAndMeanEquicontinuity}.
     \item All the torus maps in the above construction, and hence also the
       resulting diffeomorphisms $\varphi$, may be chosen as the projective
       actions of quasiperiodic $\sltr$-cocycles (compare
       \cite{haro/puig:2006}).  \listend
 \end{rem}

 \section{Total strict ergodicity for lifts and non-existence of additional eigenvalues} \label{NewExamples}

Given any $l,m\in\N$, $\T^{(l,m)}=\R/l\Z \times \R/m\Z$ is a canonical finite
covering space of the torus $\torus$. For any torus homeomorphism $\psi$
homotopic to the identity, there exist lifts $L_{(l,m;s)}^\psi
:\T^{(l,m)}\to\T^{(l,m)}$ with $s\in A(l,m)=(\Z/l\Z)\times(\Z/m\Z)$, which are
uniquely determined by the requirement that $L_{(l,m;s)}^\psi(0)\in
[s_1,s_1+1)\times[s_2,s_2+1)$. Note that two different lifts $L^\psi_{(l,m;s)}$
    and $L^\psi_{(l,m;s)}$ are conjugate by the integer translation
    $(x,y)\mapsto (x+(s_1'-s_1),y+(s_2'-s_s))$ on $\T^{(l,m)}$ and thus share
    the same dynamical properties.

      We denote the iterates of these lifts by
      $L_{(l,m;s)}^{\psi,j}=\left(L_{(l,m;s)}^\psi\right)^j$. Note that
      $L_{(l,m;s)}^{\psi,j}$ may differ from $L_{(l,m;s)}^{\psi^j}$ by an
      integer translation.  For any rotation $R_\rho$ on \torus, its lift
      $L_{(l,m;s)}^{R_\rho}$ is conjugate to the torus rotation
      $R_{((\rho_1+s_1)/l,(\rho_2+s_2)/m)}$, where a conjugacy $h_{(l,m)}$ is
      simply given by rescaling, that is, $h_{(l,m)}(x,y)=(x/l,y/m)$. Obviously,
      this does not affect the arithmetical properties of the rotation vector
      (rational, irrational or totally irrational). We can also rescale the
      lifts $L^\psi_{(l,m;s)}$ to obtain homeomorphisms
      $\ell^\psi_{(l,m;s)}=h_{(l,m)}\circ L^\psi_{(l,m;s)}\circ h^{-1}_{(l,m)}$
      of the standard torus \torus.  Further, given torus homeomorphisms
      $\varphi,\psi$ homotopic to the identity, we have
      \begin{equation}
      d_k(\varphi,\psi) \ = \ \min_{s\in
        A(l,m)}d_k\left(L^\varphi_{(l,m;0)},L^\psi_{(l,m;s)}\right) \ \geq
      \ \min_{s\in A(l,m)}d_k\left(\ell^\varphi_{(l,m;0)},\ell^\psi_{(l,m;s)}\right)
      \end{equation}

      \subsection{Total strict ergodicity of the lifts}
      
      In order to ensure that all iterates of all lifts of the diffeomorphism
      $\varphi$ from Theorem~\ref{t.meanequi_with_full_fibres} are strictly
      ergodic as well, we may now modify the construction in
      Section~\ref{MainConstruction} by replacing conditions
      (\ref{e.mainconstr1})--(\ref{e.mainconstr3}) with the following stronger
      assumptions: we recursively choose sequences $s^{(l,m)}_n,\tilde
      s^{(l,m)}_n,\hat s^{(l,m)}_n$ such that
      \begin{eqnarray*}
        d_n\left(L^{\varphi_n}_{\kl l,m,s^{(l,m)}_n\kr},L^{\tilde\varphi_n}_{\kl
          l,m,\tilde s^{(l,m)}_n\kr}\right) & =
        &d_n\kl\varphi_n,\tilde\varphi_n\kr \ , \\ d_n\kl
        L^{\tilde\varphi_n}_{\kl l,m,\tilde
          s^{(l,m)}_n\kr},L^{\hat\varphi_n}_{\kl l,m,\hat s^{(l,m)}_n\kr}\kr & =
        & d_n(\tilde\varphi_n,\hat\varphi_n) \ , \\ d_n\kl
        L^{\hat\varphi_n}_{\kl l,m,\hat s^{(l,m)}_n\kr},L^{\varphi_{n+1}}_{\kl
          l,m,s^{(l,m)}_{n+1}\kr}\kr & = & d_n(\hat \varphi_n,\varphi_{n+1})
      \end{eqnarray*}
hold for all $n\in\N$. Then, in the $n$-th step of the induction, we exert
control over the speed of convergence not only for the original maps
$\varphi_n,\hat\varphi_n,\tilde\varphi_n$, but also for all iterates of all
lifts up to level $n$. To that end, we require that for all $l,m,j=1\ld n$ we
have
       \begin{eqnarray}
  \lefteqn{ d_n\kl
    \ell^{\varphi_n,j}_{\kl l,m;s^{(l,m)}_n\kr},\ell^{\tilde\varphi_{n},j}_{\kl l,m;\tilde
      s^{(l,m)}_n\kr}\kr} \nonumber \\ & \leq &
  \drittel\label{e.modified_mainconstr1}
  \min\left\{\eta^\mathrm{se}_{n-1}\kl\ell^{\hat\varphi_1,j}_{\kl l,m;\hat
    s^{(l,m)}_1\kr}\ld \ell^{\hat\varphi_{n-1},j}_{\kl l,m;\hat
    s^{(l,m)}_{n-1}\kr}\kr,\eta^\mathrm{me}_{n-1}\kl\ell^{\tilde\varphi_1,j}_{\kl l,m;\tilde
    s^{(l,m)}_1\kr}\ld \ell^{\tilde\varphi_{n-1},j}_{\kl l,m;\tilde
    s^{(l,m)}_{n-1}\kr}\kr\right\}  \\ \nonumber
  \lefteqn{d_n\kl\ell^{\tilde\varphi_n,j}_{\kl l,m;\tilde
      s^{(l,m)}_n\kr},\ell^{\hat\varphi_{n},j}_{\kl l,m;\hat s^{(l,m)}_n\kr}\kr}
  \\ \label{e.modified_mainconstr2} & \leq & \drittel
  \min\left\{\eta^\mathrm{se}_{n-1}\kl\ell^{\hat\varphi_1,j}_{\kl l,m;\hat
    s^{(l,m)}_1\kr}\ld \ell^{\hat\varphi_{n-1},l}_{\kl l,m;\hat
    s^{(l,m)}_{n-1}\kr}\kr,\eta^\mathrm{me}_{n}
  \kl\ell^{\tilde\varphi_1,l}_{\kl l,m;\tilde s^{(l,m)}_1\kr}\ld
  \ell^{\tilde\varphi_{n},j}_{\kl l,m;\tilde s^{(l,m)}_{n}\kr}\kr\right\} \quad\quad
  \ \\ \nonumber \lefteqn{ d_n\kl\ell^{\hat\varphi_n,j}_{\kl l,m;\hat
      s^{(l,m)}_n\kr},\ell^{\varphi_{n+1},j}_{\kl l,m;s^{(l,m)}_{n+1}\kr}\kr}
  \\ \label{e.modified_mainconstr3} & \leq & \drittel
  \min\left\{\eta^\mathrm{se}_{n}\kl\ell^{\hat\varphi_1,j}_{\kl l,m;\hat
    s^{(l,m)}_1\kr}\ld \ell^{\hat\varphi_{n},j}_{\kl l,m;\hat
    s^{(l,m)}_n\kr}\kr,\eta^\mathrm{me}_{n}\kl\ell^{\tilde\varphi_1,j}_{\kl l,m;\tilde
    s^{(l,m)}_1\kr}\ld \ell^{\tilde\varphi_{n},j}_{\kl l,m;\tilde
    s^{(l,m)}_n\kr}\kr\right\}
       \end{eqnarray}

    With the same reasoning as in Section~\ref{MainConstruction}, we now obtain
    that for any $(l,m)\in\N^2$ and $j\in \N$ the sequence
    $\ell^{\hat\varphi_n,j}_{\left(l,m,\hat s^{(l,m)}_n\right)}$ converges to a
    strictly ergodic diffeomorphism, which is a rescaled lift
    $\ell^{\varphi,j}_{(l,m,s)}$ of an iterate of $\varphi=\nLim \hat\varphi_n$.

    Let $\psi=\ell_{(1,m;0)}^\varphi$. The invariant measure of $\varphi$ is of
    the form $\mu=(\Id_{\kreis}\times\gamma)_*\Leb_{\kreis}$, where
    $\gamma:\kreis\to\kreis$ is the measurable function whose graph supports
    $\mu$. Consequently, if we let
    \begin{equation}
      \gamma_j:\kreis\to\kreis \quad ,\quad  x\mapsto
    \frac{\gamma(x)+j-1}{m} \ ,\quad j=1\ld m, 
    \end{equation}
    then
    \begin{equation}
      \label{e.mupsi}
      \mu^\psi=\mtel\jmsum (\Id_{\kreis}\times\gamma_j)_*\Leb_{\kreis}
    \end{equation}
    defines an invariant measure for the rescaled lift $\psi$. By unique
    ergodicity, it is the only $\psi$-invariant measure. This proves assertions
    (a)--(c) of Theorem~\ref{t.m:1-topomorphic_extensions}.

    \subsection{Non-existence of additional eigenvalues}

    We consider a torus diffeomorphism $\varphi$ that satisfies the assertions
    (a)--(c) of Theorem~\ref{t.m:1-topomorphic_extensions}. Fix $m\in\N$ and let
    $\psi=\ell^\varphi_{(1,m;0)}$ as above. Recall that both mappings are skew
    products over the irrational rotation $r_\alpha:x\mapsto x+\alpha$. Our aim
    is to show that $\psi$ has the same discrete dynamical spectrum as
    $\varphi$, that is, there exist no additional dynamical eigenvalues for
    $\psi$.

    Suppose that $\gamma:\kreis\to\kreis$ is the measurable function whose graph
    supports the unique $\varphi$-invariant measure $\mu$, that is,
    $\mu=(\Id_{\kreis}\times\gamma)_*\Leb_{\kreis}$. Then, as discussed in the
    previous section, the unique $\psi$-invariant measure $\mu^\psi$ is given by
    (\ref{e.mupsi}).  Now, suppose for a contradiction that $f\in
    L^2_{\mu^\psi}(\torus)$ is an eigenfunction of $U_\psi$ with a new
    eigenvalue $\lambda$ that is not contained in the group of eigenvalues
    $M(\alpha)=\{\exp(2\pi ik\alpha)\mid k\in\Z\}$ of $U_\varphi$. Then $f$
    cannot be constant in the fibres (that is, independent of the second
    coordinate $y$), since in this case $x\mapsto f(x,0)$ would define an
    eigenfunction of $r_\alpha$ with eigenvalue $\lambda$, contradicting the
    fact that the eigenvalue group of $r_\alpha$ is $M(\alpha)$ as
    well. Further, the function
    \[
    g:\kreis\to\kreis \quad , \quad x \mapsto \prod_{j=1}^m f(x,\gamma_j(x))
    \]
    is an eigenfunction of $r_\alpha$ with eigenvalue $\lambda^m$, since we have
    \[
    \psi(\{(x,\gamma_1(x))\ld (x,\gamma_m(x))\}) \ =
    \ \{(x+\alpha,\gamma_1(x+\alpha))\ld (x+\alpha,\gamma_m(x+\alpha))\}
    \]
    and therefore 
    \[
      g(x+\alpha) \ = \ \prod_{j=1}^m f\circ\psi(x,\gamma_j(x)) \ =
      \ \lambda^m\prod_{j=1}^m f(x,\gamma_j(x)) \ = \ \lambda^m g(x) \ 
      \]
    $\Leb_{\kreis}$-almost surely on \kreis. Hence, $g$ is an eigenfunction of
      $r_\alpha$.  This implies $\lambda^m=\exp(2\pi ik\alpha)$ for some
      $k\in\Z$, so that $\lambda$ must be of the form
    \[
    \lambda \ = \ \exp\left(2\pi i\left(\frac{k\alpha+p}{m}\right)\right)
    \]
    for some $(k,p)\in(\Z\times\{0\ld m-1\})\smin (m\Z\times\{0\})$.

    We first assume that $k=0$. In this case, $f$ is an eigenfunction of
    $\psi^m$ with eigenvalue $\lambda^m=1$. As $f$ is non-constant, this
    contradicts the ergodicity of $\psi^m$.

    Secondly, assume that $p=0$. In this case, we consider the rescaled lift
    $\tilde\psi=\ell^\psi_{(1,m,0)}$ of $\psi$, which is now a skew product over
    the rotation $r_{\alpha/m}$. The eigenfunction $f$ transforms to an
    eigenfunction
    \[
    \tilde f(x,y) \ = \ f(mx,y)
    \]
    of $U_{\tilde\psi}$, which still has the same eigenvalue $\lambda=\exp(2\pi
    ik\alpha/m)$. However, this is now an eigenvalue of the underlying rotation
    $r_{\alpha/m}$, which corresponds to the eigenfunction $g(x,y)=\exp(2\pi
    ikx/m)$ of $U_{\tilde\psi}$. As $g$ is constant in $y$ for $\nu$-almost
    every $x$, but $\tilde f$ is not, the two eigenfunctions cannot
    coincide. Since they have the same eigenvalue, this contradicts the unique
    ergodicity of $\tilde\psi$ (note that $\tilde\psi$ is still a lift of the
    original map $\varphi$ and is therefore uniquely ergodic).\medskip

    Finally, we consider the case where $k\neq 0\neq p$. In this case, $f$ is an
    eigenfunction of $\psi^m$, with eigenvalue $\exp(2\pi ik\alpha)$. However,
    $\psi^m$ is a rescaled lift of $\varphi^m$, which has the same properties as
    $\varphi$ (it satisfies the assertions of
    Theorem~\ref{t.meanequi_with_full_fibres}), but has underlying rotation
    number $m\alpha$. This means that we are in exactly the same situation as in
    the case $p=0$ above, and again arrive at a contradiction.\medskip

    Altogether, this shows that $\psi$ has exactly the same dynamical
    eigenvalues as $\varphi$. However, the two systems cannot be isomorphic, as
    measure-theoretic factor maps into group rotations are uniquely determined up to post-composition
    with a rotation and the canonical factor map from $\psi$ to $\varphi$ is
    $m$:$1$. Due to the Halmos-von Neumann Theorem, $\psi$ and $\varphi$ cannot
    have the same (purely discrete) dynamical spectrum. This means that the
    spectrum of $U_\psi$ must have a continuous component.

    \subsection{Singularity of the continuous spectral component}
    \label{SingularSpectrum}

In order to complete the proof of Theorem~\ref{t.m:1-topomorphic_extensions},
our aim now is to show that the Anosov-Katok construction of $\varphi$ can be
modified such that the map $\psi$ defined in the last section has a singular
continuous spectral component. To that end, we need to show that $\varphi$ and
all its lifts admit cyclic approximation by periodic transformations with speed
$o(1/n)$, in the sense of Theorem~\ref{t.periodic_approximations}. The main
problem here lies in the fact that -- unlike in Anosov-Katok construction in an
area-preserving setting -- the unique invariant measure $\mu$ of the
transformation $\varphi$ is not known {\em a priori}. Therefore, it is necessary
to control both the size of the symmetric differences between the images of
partition elements under $\varphi_n$ and the eventual limit $\varphi$ and also
the limit measure of these sets at the same time. Recall that, in the end, we
need to show that there exist suitable partitions $\cP_n$ that satisfy
conditions (P1)--(P3) from Section~\ref{DynamicalSpectrum}. This will exclude
the existence of an absolutely continuous spectral component and thus complete
the proof. 
\medskip

We adopt the notation from the main construction in
Section~\ref{MainConstruction}. In particular, $q_n$ is the denominator of the
rotation vector $\rho_n$ of the $n$-th approximating diffeomorphism
$\varphi_n$. Note that $\rho_n$ was chosen only after the $n$-th conjugating
diffeomorphism $H_n$ was defined. Hence, we can require that
\begin{equation} \label{e.partition_diameter}
  d(x,y)<2/q_n \ \follows \ d(H_n(x),H_n(y))<1/n \ . 
\end{equation}
We define the partition $\cP_n$ as $\cP_n=\{P_{n,i,j}\mid i,j=0\ld q_n-1\}$,
where
\[
P_{n,i,j}=H_n\left([i/q_n,(i+1)/q_n)\times [j/q_n,(j+1)/q_n)\right) \ .
\]
Note that $\varphi_n=H_n\circ R_{\rho_n}\circ H_n^{-1}$ cyclically permutes the
elements of $\cP_n$ due to the fact that $\rho_n=(p_n/q_n,p_n'/q_n)$ with
$p_n,p_n',q_n$ relatively prime (\ref{e.relatively_prime}). Moreover, due to
(\ref{e.partition_diameter}), the maximal diameter of an element of $\cP_n$ is
at most $1/n$, which implies (P2) due to the regularity of the measure
$\mu$. Hence, both (P1) and (P2) are satisfied.\smallskip

It remains to show (P3) with sufficiently fast speed of convergence. We choose
an arbitrary function $s:\N\to\R^+$ which satisfies $\nLim ns(n)=0$, so that
Theorem~\ref{t.periodic_approximations} will be applicable. As
$K_n=\sharp\cP_n=q_n^2$, we have to ensure that
\begin{equation}
  \label{e.partition_approximation_speed}
  \sum_{i,j=0}^{q_n-1} \mu\left( \varphi_n\kl P_{n,i,j}\kr \triangle \varphi\kl
  P_{n,i,j}\kr\right) \ \leq \ s(q_n^2) \ .
\end{equation}
In order to do so, we need to introduce further inductive assumptions into the
construction carried out in Section~\ref{MainConstruction} that we already
modified by (\ref{e.modified_mainconstr1})--(\ref{e.modified_mainconstr3})
above.

Suppose that $n\in\N$ and $H_{n+1}$ has already been chosen, but not the
rotation vectors $\tilde\rho_n,\ \hat\rho_n$ and $\rho_{n+1}$ (which then define
$\tilde\varphi_n,\ \hat\varphi_n$ and $\varphi_{n+1}$). Let $\mu_n=\kl
H_{n+1}\kr_*\Leb_{\torus}$ and note that independent of the choice $\hat\rho_n$
(assuming total irrationality), this is the unique invariant measure of
$\hat\varphi_n=H_{n+1}\circ R_{\rho_n}\circ H^{-1}_{n+1}$. For $i,j=1\ld q_n$,
we choose continuous functions $f_{n,i,j}:\torus\to[0,1]$ such that
\[
\partial \left(\varphi_n\left(P_{n,i,j}\right)\right) \ \ssq
\ \inte\left(f_{n,i,j}^{-1}(1)\right) \ 
\]
and
\[
  \int_{\torus} f_{n,i,j} d\mu_{n} \ < \ s(q_n^2)/q_n^2 \ .
  \]
  For the latter condition, note that since $\varphi_n$ simply permutes the
  elements of $\cP_n$, the set $\partial
  \left(\varphi_n\left(P_{n,i,j}\right)\right)$ is simply the boundary of
  another partition element, and therefore a smooth curve that has measure zero
  with respect to $\mu_n$ (which has smooth density with respect to Lebesgue,
  since it is the image of the Lebesgue measure under the smooth diffeomorphism
  $H_n$).

   If $\tilde\rho_n$ and subsequently $\hat\rho_n$ are chosen sufficiently close
   to $\rho_n$, so that $\tilde\varphi_n$ and $\hat\varphi_n$ are close to
   $\varphi_n$, then we have
   \begin{equation}\label{e.singular_construction_1}
\hat\varphi_n\left(P_{n,i,j}\right) \Delta \varphi_n\left(P_{n,i,j}\right)
\ \ssq \ \inte\left(f_{n,i,j}^{-1}(1)\right)
   \end{equation}
   By unique ergodicity, there exists $M_n\in\N$ such that
  \begin{equation} \label{e.singular_construction_2}
  \sup_{x\in\torus} \frac{1}{M_n} \sum_{l=1}^{M_n} f_{n,i,j}\circ
  \hat\varphi_n^l(x) \ < \ s(q_n^2)/q_n^2 \ .
  \end{equation}
 Since both (\ref{e.singular_construction_1}) and
 (\ref{e.singular_construction_2}) are open conditions, we may now choose
 $\delta_n>0$ such that, for all $\psi\in
 \overline{B_{\delta_n}(\hat\varphi_n)}$, the following conditions hold.
 \begin{eqnarray}\label{e.singular_construction_3}
   \psi\left(P_{n,i,j}\right) \Delta \varphi_n\left(P_{n,i,j}\right) & \ssq &
   \inte\left(f_{n,i,j}^{-1}(1)\right) \\ \sup_{x\in\torus} \frac{1}{M_n}
   \sum_{l=1}^{M_n} f_{n,i,j}\circ \psi^l(x) & < & s(q_n^2)/q_n^2
   \ .\label{e.singular_construction_4}
 \end{eqnarray}
 
 We can now require, throughout the inductive construction in
 Section~\ref{MainConstruction}, that for all $n\in\N$ we have
 \begin{equation}
   d_0\kl\hat\varphi_n,\hat\varphi_m\kr \ \leq \ \delta_m \quad\textrm{ for all }
   m=1\ld n-1 \ . 
 \end{equation}
 For this, when going from $n$ to $n+1$, it suffices to ensure that
 \[
\max\left\{ d_0\kl\hat\varphi_n,\varphi_{n+1}\kr,\ d_0\kl
\varphi_{n+1},\tilde\varphi_{n+1}\kr,\ d_0\kl\tilde
\varphi_{n+1},\hat\varphi_{n+1}\kr\right\} \ < \ \drittel \min_{m=1}^n \delta_m-d_0\kl
\hat\varphi_n,\hat\varphi_m\kr \ .
 \]
 This, in turn, is simply achieved by a sufficiently small variation of the
 rotation vectors when choosing $\rho_{n+1},\tilde\rho_{n+1}$ and
 $\hat\rho_{n+1}$. In particular, it does not contradict any other recursive
 assumptions that we have made elsewhere during the construction.

 As a consequence, the resulting limit $\varphi$ will still
 satisfy~(\ref{e.singular_construction_3}) and (\ref{e.singular_construction_4})
 (with $\psi$ replaced by $\varphi$). However, if $\mu$ denotes the unique
 $\varphi$-invariant measure, then the above conditions imply that, for all
 $n\in\N$,
 \[
 \mu\left(\varphi\left(P_{n,i,j}\right) \Delta
 \varphi_n\left(P_{n,i,j}\right)\right)
 \ \stackrel{(\ref{e.singular_construction_3})}{\leq} \ \int_{\torus} f_{n,i,j}
 d\mu \ \stackrel{(\ref{e.singular_construction_4})}{\leq} \ s(q_n^2)/q_n^2 \ .
 \]
 This proves (\ref{e.partition_approximation_speed}), so that
 Theorem~\ref{t.periodic_approximations} yields the absence of singular
 continuous spectrum for $U_\varphi$. Hence, assertion (d) of
 Theorem~\ref{t.m:1-topomorphic_extensions} holds, which completes the proof.

\end{document}